\documentclass[a4paper, 10pt]{article}
\usepackage[utf8]{inputenc}

\usepackage{color}

\usepackage{amsmath} 
\usepackage{amssymb} 
\usepackage{amsthm}
\usepackage{graphicx}
\usepackage[colorlinks=true,linkcolor=blue]{hyperref}
\usepackage{hyphenat}

\theoremstyle{plain}

\newtheorem{prop}{Proposition}[section]

\newtheorem{lem}{Lemma}[section]

\theoremstyle{remark}
\newtheorem{rmk}{Remark}

\newcommand {\R} {\mathbb{R}} 
 \newcommand {\N} {\mathbb{N}}

\newcommand {\D} {\Delta}

\newcommand {\dt} {\partial_t}

\newcommand{\va}{\varphi}
\newcommand{\Om}{\Omega}

\DeclareGraphicsExtensions{.jpg,.pdf,.pdftex,.eps,.png, .ps}

\DeclareMathOperator {\supp} { supp}

\newcommand{\La}{\mathcal{L}}
\begin{document}

\title{On the Backward Uniqueness Property for the Heat Equation in Two-Dimensional Conical Domains}
\author{Angkana R\"uland\footnote{This work is part of the PhD thesis of the author written under the supervision of Prof. Herbert Koch to whom she owes great gratitude for his persistent support and advice. She thanks the Deutsche Telekomstiftung and the Hausdorff Center for Mathematics for financial support. }\footnote{Mathematisches Institut, Universit\"at Bonn, Endenicher Allee 60, 53115 Bonn, Germany, rueland@math.uni-bonn.de, +49/228/7362259 \newline \textit{MSC (2010)}: Primary 35K05, Secondary 35R45, 35A02.}}
\maketitle

\begin{abstract}
In this article we deal with the backward uniqueness property of the heat equation in conical domains in two spatial dimensions via Carleman inequality techniques. Using a microlocal interpretation of the pseudoconvexity condition, we improve the bounds of \v{S}ver\'ak and Li \cite{LiSv} on the minimal angle in which the backward uniqueness property is displayed: We reach angles of slightly less than $95^{\circ}$. Via two-dimensional limiting Carleman weights we obtain the uniqueness of possible controls of the heat equation with lower order perturbations in conical domains with opening angles larger than $90^{\circ}$.
\end{abstract}

\tableofcontents

\section{Introduction}

In the sequel we will be concerned with controllability properties of the heat equation. We will focus on the so called ``backward uniqueness property" for the heat equation. This deals with the question of whether
\begin{equation}
\begin{split}
\label{eq:BUCP}
(\dt - \D)u &= Vu+ W\cdot \nabla u \mbox{ in }\Omega\times (0,1),\\
u(t=1,x) &= 0 \mbox{ in } \Omega,
\end{split}
\end{equation}
already implies $u\equiv 0$ in $\Omega \times (0,1)$ for appropriate choices of the potentials $V$ and $W$. In particular,  the validity of the backward uniqueness property entails that there are no nontrivial initial and boundary data such that $u$ satisfies (\ref{eq:BUCP}). Due to the linearity of the heat equation such a property can be interpreted ``causally": Only a single possible choice of data can lead to a specific final state of a system if it is evolved by the heat equation. In other words the ``final state determines its past". Such a property would, for example, entail that if the temperature distributions of two objects agree at a given point in time, the history of the temperature distributions must have been identical at all previous times. From physical experience, e.g. heating a plate, one would not expect such a behaviour (for objects of finite size).\\

Controllability properties of the heat equation have been thoroughly investigated in \emph{bounded} domains, c.f. \cite{LR}, \cite{Z1}, \cite{Z2}, \cite{FR}, \cite{Rus}: In this case the heat equation is null-controllable, i.e. any $L^2$ initial datum can be driven to zero. In this situation there are various approaches relying on Carleman estimates, eigenfunction estimates, the method of moments and observability inequalities. \\
In the case of \emph{unbounded} domains the problem is less transparent. It is known that there are two regimes which both distinctly different from the setting in bounded domains as the equation is no longer null-controllable: 
\begin{itemize}
\item In the case of ``small" angles ($\theta < \frac{\pi}{2}$) there exist initial data which can be driven to zero (c.f. Section \ref{sec:examples}). 
\item In the ``large" angle regime the equation features the backward uniqueness property. This means that it becomes impossible to diffuse the information from the boundary into the interior sufficiently fast.
\end{itemize}

In the sequel we concentrate on the large angle regime. The understanding of this regime is rather incomplete: Although heuristics suggest that the backward uniqueness property should hold for all angles larger than (or equal to) $\theta= \frac{\pi}{2}$, this has not been proven. In part, this lack of a full understanding certainly originates from the difficulty of deriving sufficiently strong \emph{lower} bounds for the solution of the heat equation. All approaches have to rule out possible oscillation and cancellation effects. In spite of this incomplete picture, there are various partial results, c.f. \cite{LiSv}, \cite{MZ1}, \cite{MZ2}, \cite{M}.  The strongest result is given by Li and \v{S}ver\'ak \cite{LiSv} who employ Carleman techniques in order to derive the backward uniqueness property for the heat equation in conical domains with opening angles of down to approximately $109^{\circ}$. However, the underlying Carleman weight does not have sufficient convexity properties in order to carry the estimate beyond this number.\\

The aim of the present article is to further study the ``large" angle regime and to derive better bounds for the critical angle in two space dimensions.
Investigating the necessary  properties of weight functions of Carleman estimates, it is possible to give a condition guaranteeing pseudoconvexity -- i.e. admissibility -- for a larger class of weight functions in two dimensions. With these it is possible to reach angles of approximately $95^{\circ}$. \\

The guiding intuition behind these estimates is provided by the time\hyp independent case: For lower order perturbations of the Laplacian, Carleman estimates hold down to an angle of $90^{\circ}$ in the two-dimensional case. Thus, these estimates provide backward uniqueness properties for the heat equation if additionally $u(0,\cdot)=0$ is assumed (c.f. Proposition \ref{prop:backwaruniquness-}).
The general case, however, is much more difficult to handle, as any Carleman estimate for the heat equation in conical domains can be thought of a deformation of an elliptic weight (at time frequency zero).

\subsection{Main Results}
\label{sec:results}

In the sequel we prove the backward uniqueness property of the heat equation in conical domains with opening angles larger than $95^{\circ}$. As in the results of \v{S}ver\'ak et al. \cite{LiSv}, \cite{ESS}, this property is a consequence of an application of the following Carleman inequality in two spatial dimensions. 
Using the notation
\begin{align*}
\Omega_{\theta}= \left\{ (x_{1},x_{2})\subset \R^2 \Big| \tan(\theta/2)\geq \frac{|x_{2}|}{x_{1}}, \; x_{1}\geq 0 \right\} \subset \R^2,
\end{align*}
we have:
\begin{prop}[Carleman Estimate]
\label{prop:Carl}
Let $u\in C_{0}^{\infty}(\Omega_{\theta}\setminus B_{R}(0))$, $R\gg 1$ sufficiently large, $\theta \geq 95^{\circ}$. Then there exists a Carleman weight
 $\phi(t,x)$, $|\phi(t,x)|< C \frac{|x|^2}{t}$ 
such that
\begin{align}
\label{eq:Carlm}
\tau \left\| e^{\tau \phi}\frac{(1-t)^{\frac{1}{2}}}{t} u \right\|_{L^2}+ \tau^{\frac{1}{2}}\left\| e^{\tau \phi} u \right\|_{L^2}+\left\| e^{\tau \phi}\nabla  u \right\|_{L^2} \lesssim \left\| e^{\tau \phi}(\dt + \D)u \right\|_{L^{2}}.
\end{align}
\end{prop}

\begin{figure}[h]
\centering
\includegraphics{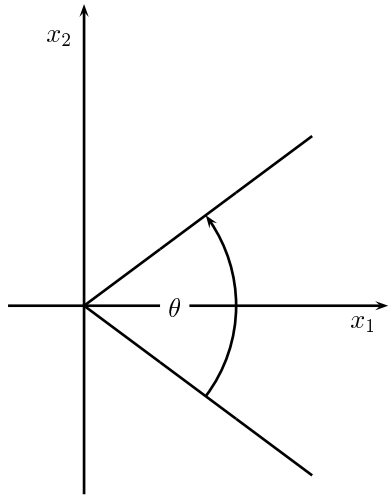}
\caption{The domain $\Omega_{\theta}$.}
\end{figure}

The difficulty in proving this estimate stems from the loss of convexity of the weight function. Due to the restrictions on its radial growth (which is necessary if the inequality is to be applied to the backward uniqueness problem), it cannot be easily convexified in the radial direction which would simplify the proof of the Carleman inequalities significantly. \\
As in \cite{LiSv}, the backward uniqueness property is a direct consequence of the Carleman estimate:

\begin{prop}[Backward Uniqueness of the Heat Equation in Angular Domains]
\label{prop:backwarduniqueness}
Let $\theta \geq 95^{\circ}$ and assume that $u: [0,1]\times \Omega_{\theta}  \rightarrow \R$ satisfies 
\begin{equation}
\begin{split}
\label{eq:heat}
|(\dt + \D)u| &\leq C(|u|+|\nabla u|) \mbox{ in } [0,1]\times \Omega_{\theta},\\
u(0,x) &= 0 \mbox{ in } \Omega_{\theta},\\
|u| & \leq M \mbox{ in } [0,1]\times \Omega_{\theta}.
\end{split}
\end{equation}
Then $u\equiv 0$.
\end{prop}

\begin{rmk}
The angle $\theta \geq 95^{\circ}$ is not optimal. Various numerical experiments suggest that evaluating the one-dimensional pseudoconvexity condition, i.e. expression (\ref{eq:pseudo}), it is possible to reach angles of less than $95^{\circ}$. However, the gain seems to be marginal (one reaches angles of $\sim 94.8^{\circ}$); in fact, it seems not easy to reach angles of less than $94^{\circ}$ (via one-dimensional weight functions).
\end{rmk}

Under the additional assumption that both the initial and final data vanish, it is possible to prove the backward uniqueness property in angles strictly larger than $90^{\circ}$ in two dimensions. This is a nontrivial result depending strongly on the \textit{unboundedness} of the underlying domain. In fact, in any \emph{bounded} domain it would be possible to find a large variety of boundary controls satisfying the initial and final condition. \\
As in the dissertation of Li \cite{Li}, this (conditional) uniqueness statement is a consequence of decay properties of the underlying \textit{elliptic} problem: In domains of opening angles greater than or equal to $90^{\circ}$ there are no harmonic functions decaying with a Gaussian rate. Instead of employing the Phragmen-Lindel\"of theorem for harmonic functions (as Li does), we argue via an elliptic Carleman estimate for which we use a limiting Carleman weight in the sense of Kenig et al. \cite{KSU}. Compared with the Phragmen-Lindel\"of-Ansatz this strategy seems to be more stable and allows to include lower order perturbations (with time independent coefficients).

\begin{prop}[Uniqueness of the Control Function in $L^{2}$]
\label{prop:backwaruniquness-}
Let $\theta>\frac{\pi}{2}$, $\alpha = \frac{\pi}{\theta}$ and assume that  $u: [0,1]\times \Omega_{\theta} \rightarrow \R$ satisfies
\begin{equation}
\begin{split}
\label{eq:heat-}
(\dt + \D)u &= c_{1}(x)u + c_{2}(x)\cdot \nabla u \mbox{ in } [0,1]\times \Omega_{\theta},\\
|c_{2}(x)| & \leq C\frac{1}{|x|^{b(\theta)}} \mbox{ in } \Omega_{\theta},\\
u(1,x) &= 0 \mbox{ in } \Omega_{\theta},\\
u(0,x) &= 0 \mbox{ in } \Omega_{\theta},\\
|u|& \leq M \mbox{ in } [0,1]\times \Omega_ {\theta},
\end{split}
\end{equation}
where $b(\theta)> \frac{2-\alpha}{\alpha}$ and $c_{1}\in L^{\infty}$.
Then $u\equiv0$.
\end{prop}


\begin{rmk}
Proposition \ref{prop:backwaruniquness-} demonstrates the uniqueness of the possible control function for the heat equation in an unbounded, conical domain of opening angle $\theta> \frac{\pi}{2}$. This is in sharp contrast with the results for the heat equation in bounded domains, in which case there are infinitely many possibilities for such controls \cite{LRob}.
\end{rmk}

\begin{rmk}
Very recently, J. Wu and W. Wang \cite{WW} independently proved similar results reaching an angle of $99^{\circ}$. Our methods additionally rely on a detailed pseudoconvexity analysis to reach $95^{\circ}$. This allows us to identify a necessary and (almost) sufficient condition for a certain class of weight functions, c.f. Section \ref{sec:proII}.
\end{rmk}

\begin{rmk}
Matsaev and Gurarii \cite{GM} claim that the backward uniqueness result for the pure heat equation can be reduced to an existence result for the Laplacian even if no additional assumption on the behavior of $u(1,x)$ is made. However, there seems to be no proof of this statement in the literature.
\end{rmk}

Let us comment on the organization of the remainder of the article: In the following two subsections we recall characteristic examples from the literature. Then we carry out the decisive parabolic pseudoconvexity analysis, deduce an ordinary differential inequality and evaluate this numerically, c.f. Section \ref{sec:proII}. In Sections \ref{sec:Carl} and \ref{sec:bur} this is then used to prove the desired Carleman estimate from Proposition \ref{prop:Carl} and to deduce the backward uniqueness result in two-dimensional conical domains with opening angles larger than approximately $95^{\circ}$. Finally, in the last section, we illustrate how the uniqueness of controls can be reduced to an elliptic Carleman estimate which we prove by using appropriate limiting Carleman weights.

\subsection{Null-Controllability and Escauriaza's Example}
\label{sec:examples}

In the literature there is good reason indicating that the behaviour of solutions of the heat equation in unbounded domains has to differ strongly from that in bounded domains.
In spite of the infinite speed of progagation, it is not reasonable to expect an observability inequality for the heat equation in unbounded domains. As Micu and Zuazua point out, a simple translation argument proves that this cannot be possible without an additional weight \cite{MZ1}. Indeed, the null-controllability property is equivalent to an observability inequality for the adjoint system:
\begin{align*}
\left\| \varphi(0) \right\|_{L^{2}(\Omega)}^2 \leq 
C \int\limits_{0}^{T}\int\limits_{\partial \Omega}\left| \frac{\partial \varphi}{\partial \nu}\right|^2 d\mathcal{H}^{n-1}(x) dt,
\end{align*}
where $\varphi$ satisfies is a solution of the adjoint heat equation
\begin{align*}
(\dt - \D)\va &= 0 \mbox{ in } (0,T) \times \Omega,\\
 \va & = 0 \mbox{ on } (0,T) \times \partial \Omega,\\
 \va & = \va_{T} \mbox{ in } \{T\} \times \Omega.
\end{align*}
Considering $\varphi_{T}\in C_{0}^{\infty}(\Omega)$, $\varphi_{T}\geq 0$, we define translations $\varphi_{k,T}(x):= \varphi_{T}(x-k)$.
Inserting these translated data and the corresponding solution of the adjoint problem into the observability inequality, one notices that
\begin{align*}
\frac{\left\| \va(0) \right\|_{L^2(\Omega)}^2}{\int\limits_{0}^{T}\int\limits_{\partial \Omega}\left| \frac{\partial \varphi}{\partial \nu}\right|^2 d\mathcal{H}^{n-1}(x) dt} \rightarrow \infty \mbox{ as } k\rightarrow \infty.
\end{align*}
This heuristic argument suggests that the heat equation behaves differently in unbounded domains. Yet, it neither excludes controllability in weighted spaces nor the existence of specific null-controllable initial data. At this point the example of Escauriaza shows that at least in the ``small angle regime" it is possible to find bounded initial and boundary data which are null-controllable. \\

We briefly recall Escauriaza's example, c.f. \cite{LiSv}: Considering the rest of the variables as dummy variables, it suffices to give an example in two dimensions only. Here, it is possible to make use of the Appell transform to switch from a solution $u(x,t)$ of the forward heat equation to a solution $v(y,s)$ of the backward heat equation.\\
The (two-dimensional version of the) transform reads
\begin{align*}
u(x,t)& = \frac{1}{4\pi t} e^{- \frac{|x|^{2}}{4t}}v\left(\frac{x}{t},\frac{1}{t}\right),\\
y&= \frac{x}{t}, s= \frac{1}{t}.
\end{align*}
Starting with a harmonic function, it becomes possible to associate a backward caloric function to it. Considering the harmonic function
\begin{align*}
h(x) = \Re e^{-(x_{1} + ix_{2})^{\alpha}}, \ \alpha >2,
\end{align*}
its Appell transform
\begin{align*}
v(x,t) = e^{\frac{|x|^{2}}{4t}}h \left(\frac{x}{t} \right)
\end{align*}
yields a solution of the backward heat equation:
\begin{align*}
\dt v + \D v = 0,\\
v(x,0) = 0.
\end{align*}
Away from the origin this function is uniformly bounded in any conical domain with opening angle $\theta\in [0,\frac{\pi}{\alpha})$. Translating in space and reflecting in time, produces a counterexample to the backward uniqueness property of the heat equation:
\begin{align*}
u(x,t) = v(x_{1}+1,x_{2}+1,1-t).
\end{align*}

\section[The Parabolic Weight Function]{The Choice of the Weight Function and Parabolic Pseudoconvexity}
\label{sec:proII}
\subsection{The Pseudoconvexity Condition}
As we are interested in proving an anisotropic Carleman inequality, we treat the temporal and spatial variables according to the parabolic scaling in the usual conjugation procedure, c.f. \cite{Tat3}, \cite{Tat4}. Using an arbitrary weight, $\phi$, and setting $u=e^{-\phi}w$, this leads to the following expression
\begin{align*}
\left\|e^{ \phi} (\D+\dt)u\right\|_{L^{2}}^{2} = \left\| (\D+ |\nabla \phi|^{2} -2  \nabla \phi\cdot \nabla  -  \D \phi  + \dt - \ \dt \phi)w  \right\|_{L^{2}}^{2}.
\end{align*}
Separation into the symmetric and antisymmetric parts yields
\begin{align*}
& \left\| (\D+ |\nabla \phi|^{2} -2 \nabla \phi\cdot \nabla  -  \D \phi  + \dt -  \dt \phi)w  \right\|_{L^{2}}^{2}\\ 
& = \left\|  (\D + |\nabla \phi|^2 -\dt\phi)w \right\|_{L^{2}}^{2} + \left\|  (\dt - 2  \nabla \phi \cdot \nabla -  \D \phi )w  \right\|_{L^{2}}^2\\
& \;\;\;\;+ \int( [\D + |\nabla \phi|^2 -\dt\phi ,\dt - 2 \nabla \phi \cdot \nabla - \D \phi ]w,w)dx.
\end{align*}
Taking the anisotropy of the equation into account (and assuming $\phi\sim \tau$), the principal symbols of these expressions read
\begin{align*}
 p^r& = -|\xi|^2+|\nabla \phi|^2,\\
 p^i &= s-2 \nabla \phi \cdot \xi,
\end{align*} 
(in a bounded domain).
As for all Carleman inequalities, it suffices to derive the estimate on the characteristic set of the principal symbol. Here, the positivity has to originate from the commutator expression. On the characteristic set the leading order terms of the spatial commutator turn into
\begin{align*}
 \{p^r,p^i\}_{x}=  4\nabla \phi \cdot \nabla^2 \phi \nabla \phi + 4 |\nabla \phi|^2 \frac{ \xi}{|\xi|} \cdot \nabla^2 \phi \frac{ \xi}{|\xi|},
\end{align*}
while the temporal commutator is of the following form
\begin{align*}
\{p^r, p^i\}_{t}=  -2\dt |\nabla \phi|^2. 
\end{align*}
As we will see in the sequel, both terms play an essential role for our analysis: 
\begin{itemize}
\item \textbf{Decay of Null-Controllable Solutions.} An equilibrium condition for the temporal and spatial commutators allows to deduce Gaussian decay for null-controllable solutions of the heat equation. This was proved by \v{S}ver\'ak et al. \cite{ESS} and can also be extended (with appropriately adapted exponents) to higher order diffusion equations. The key idea here is to employ non-convex weights in the $x$-variable which are \textit{not} weighted by the (large) prefactor $\tau$, combined with convex weights in the temporal variable which are weighted by a factor of $\tau$. Although this implies that the \textit{spatial} commutator does not induce positivity on the intersection of the characteristic sets of the symmetric and antisymmetric parts of the operator, positivity can be obtained from the \textit{temporal} part of the commutator. The uttermost, still controllable amount of non-convexity in the spatial part is determined by an equality of the scaling of the most negative commutator contributions, $\nabla \phi \cdot \nabla^2 \phi \nabla \phi$, and the strongest positive commutator contributions, $-\dt |\nabla \phi|^2$. Thanks to the strong $\tau$ weight in time, the temporal commutator provides enough positivity in this case, c.f. Lemma \ref{lem:Sverak}.
\item \textbf{Backward Uniqueness Property.} The spatial terms dictate the necessary conditions for Carleman weights which can be used in proving the backward uniqueness property. In order to treat arbitrary boundary terms, we have to truncate the weight function on the respective spatial and temporal boundaries. This, however, implies that the weight must be very small at the boundary, while it has to become very large in the (spatial and temporal) interior of the domain. This is achieved via weights with a factor $\tau$ both in their spatial and their temporal components. From this we infer the existence of a spatial regime in which the spatial commutator dominates over the temporal one due to its scaling with $\tau^3$ (the temporal part only scales with $\tau^2$). Therefore, it becomes necessary to study the spatial weight in detail.
\end{itemize}
We proceed with the analysis of the second observation. For that purpose, we consider weights of the form $\tau \phi$ instead of $\phi$. Therefore, a necessary and sufficient condition for the positivity of the commutator on the characteristic set is given by
\begin{align}
\label{eq:pseudo1}
 \{p^r,p^i\}\geq 4\tau^3 \nabla \phi \cdot \nabla^2 \phi \nabla \phi + 4\tau^3 |\nabla \phi|^2 \lambda_{\min}(\nabla^2 \phi) \geq 0,
\end{align}
where $ \lambda_{\min}(\nabla^2 \phi)$ is the smallest eigenvalue of the Hessian $\nabla^2 \phi$. As a consequence, the weight function has to be chosen such that this property is satisfied.
For convex functions $\phi$ this is always true. However, in order to prove the Carleman estimate, the weight has to be ``small" at the boundary of the domain and ``large" in the interior. 
In fact, our Carleman weight has to satisfy the following conditions:
\begin{itemize}
\item[1.)] The weight function has to vanish on the boundary of the domain (both spatially and temporally on the time slice on which the function itself is not already vanishing), and has to be strictly positive in the (spatial and temporal) interior of the domain. This can be slightly relaxed by asking for weight functions which are ``small" (instead of vanishing) on the boundaries of the domain.
As a consequence, the weight function has to be concave in the angular variable $\varphi$ (at least partially). As the pseudoconvexity condition is strictly weaker than the standard convexity notion, it is still possible to find a non-empty class of weights in domains with sufficiently large opening angles.
\item[2.)] As observed by Escauriaza, Seregin and \v{S}ver\'ak \cite{ESS} null-controllable solutions of the heat equation have Gaussian decay at infinity:

\begin{lem}[Gaussian Decay, \cite{ESS}]
\label{lem:Sverak}
Let $u:[0,T]\times B_{R}(0) \rightarrow \R$ satisfy
\begin{align*}
|\dt u + \D u| &\leq c_{1}(|\nabla u| + |u|) \mbox{ in } (0,T)\times B_{R}(0),\\
u(0,x)& = 0 \mbox{ in } B_{R}(0),\\
|u| & < M \mbox{ in } (0,T)\times B_{R}(0),
\end{align*}
for some constant $c_{1}<\infty$. Then there exist constants $\beta, \gamma$, such that for $t\in (0,\gamma)$
\begin{align*}
|u(t,0)| \leq \frac{c_{2}}{\min\{1,T \}}M e^{-\beta \frac{R^{2}}{t}},
\end{align*}
where $c_{2}=c_{2}(c_{1})$, $\gamma = \gamma(c_{1},T)$.  
\end{lem}
Microlocally, the estimate of Escauriaza, \v{S}ver\'ak and Seregin uses an equilibrium between a relatively weak, non-convex spatial weight and a very strong, convex temporal weight.\\

For the desired Carleman estimates, it in particular implies that the growth of admissible weights is restricted: Any Carleman weight, which is constructed with the aim of showing the backward uniqueness property, has to have a strictly subquadratic growth behaviour in unbounded conical domains.
\end{itemize}

\subsection[The Weight Function]{The Ansatz for the Weight Function: Necessary and Sufficient Conditions}
In analogy to the weight function of \v{S}ver\'ak and Li \cite{LiSv}, we make the ansatz
\begin{align*}
\phi(r,\varphi):= r^\alpha f(\varphi),
\end{align*}
for a two-dimensional (spatial) weight function in polar coordinates. In this case the pseudoconvexity condition, (\ref{eq:pseudo1}), can be rephrased as a homogeneous cubic ordinary differential inequality:
\begin{equation}
\begin{split}
\label{eq:pseudo}
&(\alpha-1) \alpha^3 f(\varphi)^3+\alpha (2 \alpha-1)
   f(\varphi) f'(\varphi)^2+f'(\varphi)^2 f''(\varphi)\\
&+\frac{1}{2} \left(\alpha^2 f(\varphi)^2+f'(\varphi)^2\right)\left(\alpha^2 f(\varphi)+f''(\varphi) \right.\\
&\left. -\sqrt{(\alpha-2)^2 \alpha^2 f(\varphi)^2-2 (\alpha-2) \alpha f(\varphi) f''(\varphi)+4 (\alpha-1)^2 f'(\varphi)^2+f''(\varphi)^2}\right) \geq 0.
\end{split}
\end{equation}
However, the parameter $\alpha$ cannot be chosen arbitrarily.
Lemma \ref{lem:Sverak} implies a restriction on the possible radial dependence of the weight function: $\alpha \leq 2$. Difficulties in choosing appropriate weights therefore stem from the fact that we cannot convexify the weight in the radial variable in an arbitrarily strong manner.

\subsection{\v{S}ver\'ak's Weight Function and a Modification}
\label{sec:SW}
In order to analyze possible Carleman weights, we briefly review \v{S}ver\'ak's ansatz: The weight function
\begin{align}
\label{eq:Sverak}
\phi_{Sv}(r,\varphi)=r^\alpha\left (\cos^\alpha(\varphi) - \cos^{\alpha}\left(\frac{\theta}{2}\right)\right)
\end{align}
satisfies the pseudoconvexity condition as long as the opening angle $\theta$ remains large enough: $\theta\geq \arccos(\frac{1}{\sqrt{3}}) $. The necessity of this condition can be verified by analytically checking the pseudoconvexity condition at the boundary of the domain. Indeed, \v{S}ver\'ak's weight function degenerates at the boundary although it displays robust pseudoconvexity properties in the interior  (c.f. Figure \ref{fig:Sverak}).
\begin{figure}[ht]
\centering
\includegraphics[scale=0.95]{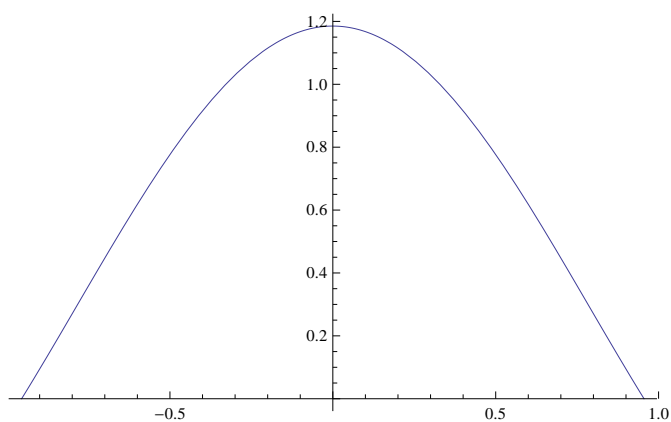}
\caption{The pseudoconvexity condition is satisfied for \v{S}ver\'ak's weight function: The $x-$axis depicts the angle in radians while we plot the values of the pseudoconvexity-expression (\ref{eq:pseudo}) on the $y$-axis. We note that the pseudoconvexity properties of the weight function degenerate at the boundary.}
\label{fig:Sverak}
\end{figure}
A limitation of \v{S}ver\'ak's weight certainly consists in choosing only a one parameter family of weights. If instead the same weight is considered with a second parameter $\beta$, e.g.
\begin{align*}
\phi_{\alpha,\beta}(r,\varphi) = r^\alpha \left (\cos^\beta(\varphi) - \cos^{\beta}\left(\frac{\theta}{2}\right)\right)
\end{align*}
the angle can be reduced significantly.\\
This ansatz has the advantage that although there are restrictions on the growth of $\alpha$ there are none on the size of $\beta$, in particular $\beta\geq 2$ is an admissible exponent. Here the weight suffices to prove the backward uniqueness property in opening angles of up to approximately $95.4^{\circ}$. The drawback of this ansatz, however, is that the pseudoconvexity condition can become fragile in the interior of the domain as well (c.f. Fig. \ref{fig:weight}).\\
This two-parameter family of weight functions is certainly not optimal. A more general ansatz for a weight function could consist of making a power series ansatz and optimizing the coefficients so as to preserve pseudoconvexity in the domain. With the weight
\begin{align*}
\phi(r,\varphi) :=r^{1.99999}( 0.987609 - 1.22053 \varphi^2 + 0.562108 \varphi^4 - 0.162117 \varphi^6 \\
+  0.0481833 \varphi^8 - 0.000001 \varphi^{10}),
\end{align*}
for example, it is possible to reach angles below $95^{\circ}$.

\begin{figure}[ht]
\centering
\includegraphics[scale=0.95]{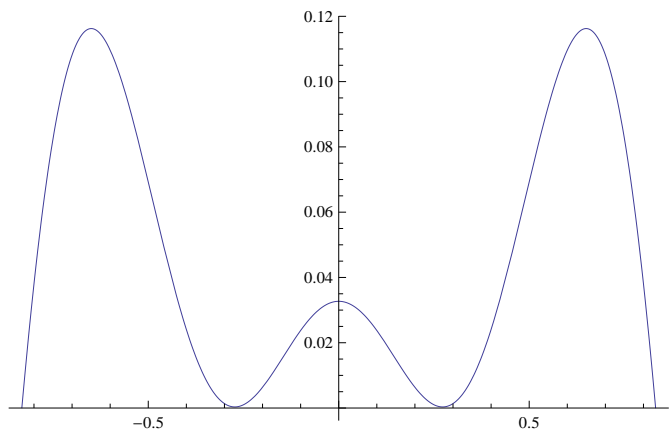}
\caption{For the angle $\theta \sim 95.4^{\circ} $ the pseudoconvexity condition is satisfied for $\phi_{\alpha,\beta}$ with $\alpha=1.999999, \beta = 2.474917$. For this weight function the pseudoconvexity condition  deteriorates at the boundary as well as in the interior.}
\label{fig:weight}
\end{figure}

\subsection{The Numerical Evaluation of the Pseudoconvexity Condition}
\label{sec:numerics}
Instead of trying to guess a suitable weight function, 
it is possible to numerically analyze the pseudoconvexity condition. As the ODE which encodes the one-dimensional pseudoconvexity condition is invariant under the reflection symmetry
\begin{align*}
f(\varphi) \mapsto f(-\varphi),
\end{align*}
we expect the solution to be symmetric if the boundary conditions are prescribed in a symmetric way. Unfortunately, the system seems to be numerically stiff; using Mathematica calculations it seems impossible to reach an angle smaller than approximately $94.8^{\circ}$  in the case of an equality in (\ref{eq:pseudo}). 
Therefore, it would be very interesting to understand the symmetric boundary value problem for (\ref{eq:pseudo}) from an analytical point of view. Due to the nonlinearity and square root in the equation this seems to be challenging.\\
Apart from these (technical) difficulties, we believe that the fundamental problem of determining admissible weights via the described one-dimensional approach is limited to approximately $95^{\circ}$. In other words, the major drawback in reaching angles closer to the conjectured $90^{\circ}$ is caused by restricting to essentially one-dimensional weight functions.

\section[Proof of the Parabolic Carleman Estimate]{Proof of the Carleman Inequality of Proposition \ref{prop:Carl}}
\label{sec:Carl}
Using the explicit weight $\phi_{1.999999, 2.474917}$, it is possible to deduce a stronger Carleman inequality and thus to prove backward uniqueness of the heat equation in conical domains with angles down to approximately $95.4^{\circ}$. \\
Once an admissible, improved weight is found, the techniques of the proof of the backward uniqueness property are not new; in fact we argue along the same lines as \v{S}ver\'ak and Li \cite{LiSv}.
As already indicated by the phase space considerations the proof has to use the pseudoconvexity properties of the weight. Although the proof will not be a phase space argument but will instead rely on a direct argument, the previous considerations form the basis of the result. Having ensured pseudoconvexity in the spatial variables, the final weight function can be chosen to have the following time dependence:
\begin{align*}
&\phi(t,r,\va)= \phi_{1}(t,r,\va) + \phi_{2}(t), \\
&\phi_{1}(r,t,\va)=\frac{1-t}{t}\phi_{1.999999, 2.474917}(r,\va), \ \phi_{2}(t)=\epsilon (1-t)^2,
\end{align*}
for a sufficiently small constant $0<\epsilon\ll 1$ to be chosen later.

\begin{proof}[Proof of Proposition \ref{prop:Carl}]
Let $u\in C_{0}^{\infty}((0,T)\times(\Omega_{\theta}\setminus B_{R}))$, $R\gg 1$. Conjugating the heat operator gives
\begin{align*}
L_{\phi} u & = (\D+ \tau^2|\nabla \phi|^{2} -2 \tau \nabla \phi\cdot \nabla  - \tau \D \phi  + \dt - \tau \dt \phi)u,\\
A_{\phi}u & = (\dt- 2\tau \nabla \phi\cdot \nabla - \tau \D\phi)u,\\
S_{\phi}u &= (\D + \tau^2|\nabla \phi|^2 -\tau \dt \phi)u.
\end{align*}
Therefore the $L^{2}$ norm of the operator turns into
\begin{align*}
\int|L_{\phi}u|^2dxdt = \int|A_{\phi}u|^2dxdt + \int|S_{\phi}u|^2dxdt + \int([S_{\phi},A_{\phi}]u,u)dxdt.
\end{align*}
The idea is to derive the lower bound by using a combination of the commutator and the symmetric part of the operator. A short calculation yields
\begin{align*}
\int ([S_{\phi},A_{\phi}]u,u)dxdt = & \ 4\int (\tau^{3}\nabla \phi \cdot \nabla^2 \phi \nabla \phi u^2+\tau \nabla u \cdot \nabla^2 \phi \nabla u)dxdt \\
& + \int (-2\tau^2 \dt|\nabla \phi|^2 +\tau  \dt^2 \phi - \tau  \D^2\phi )u^2 dxdt.
\end{align*}
As in the case of \v{S}ver\'ak and Li \cite{LiSv}, the difficulty originates from the fact that the Hessian of the weight function is not globally positive-definite (which, however, still suffices for our purposes as pseudoconvexity is a strictly weaker condition than the usual notion of convexity). Nevertheless, the numerical analysis of the pseudoconvexity properties of this weight function suggests that on the characteristic set of the symmetric and antisymmetric parts the commutator provides sufficient positivity for the Carleman inequality to hold true. In real space, this condition can be realized by deducing positivity from a combination of the commutator and the symmetric part. As \v{S}ver\'ak and Li, we introduce an auxiliary function $F(t,x)$. An integration by parts gives
\begin{align*}
\int(S_{\phi}u,Fu) dxdt = \int -F|\nabla u|^2 + (\frac{1}{2} \D F + \tau^2F |\nabla \phi|^2 - \tau \dt \phi F)u^2 dxdt.
\end{align*} 
Hence, by the binomial formula
\begin{align*}
\int (S_{\phi}u,S_{\phi}u) dxdt  \geq & \ - \int (S_{\phi}u,Fu) dxdt - \frac{1}{4}\int F^2u^2dx dt \\
\geq & \ \int F |\nabla u|^2 - (\frac{1}{2} \D F + \tau^2 F |\nabla \phi|^2 - \tau \dt \phi F)u^2 dxdt\\
&\ - \frac{1}{4}\int F^2u^2dx dt. 
\end{align*}
As in the paper of \v{S}ver\'ak and Li \cite{LiSv}, the combination of the commutator and the symmetric part yield
\begin{align*}
\int([S_{\phi},A_{\phi}]u,u) dxdt + \int|S_ {\phi}u|^2dxdt &\geq 
 \int (4\tau^{3}\nabla \phi \cdot \nabla^2 \phi \nabla \phi u^2 - \tau^2 F |\nabla \phi|^2u^2)dxdt\\
& +\int   (F|\nabla u|^2+4 \tau \nabla u \cdot \nabla^2 \phi \nabla u) dxdt\\
& + \int (-2\tau^2 \dt|\nabla \phi|^2 +\tau  \dt^2 \phi - \tau  \D^2\phi )u^2 dxdt\\
& - \int (\frac{1}{2} \D F  - \tau  \dt \phi F)u^2 dxdt \\
&- \frac{1}{4}\int F^2u^2dx dt.
\end{align*}
In order to derive positivity for the gradient term, we set
\begin{align*}
F= - 4\tau \lambda_{\min}(\nabla^2 \phi_{1}) +\frac{2}{5}.
\end{align*}
We remark that for our choice of $\phi$ the smallest eigenvalue of the Hessian of $\phi$, $\lambda_{\min}(\nabla^2 \phi)$, is a smooth function of both the angular and the radial variables if $r>0$ (c.f. Figure \ref{fig:eigen}). Thus, no additional mollification is necessary in order to deal with expressions as for instance $\D F$.
\begin{figure}[h]
\centering
\includegraphics[scale=0.95]{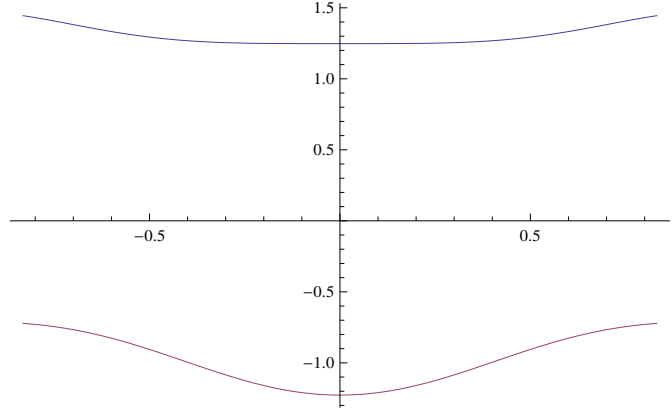}
\caption{The eigenvalues of the Hessian of the weight function depending on the angular variable $\va$ for fixed radial and temporal variables. Due to the concavity along the angular and the convexity along the radial directions, the eigenvalues have a fixed sign and do not cross. In particular, no mollification is needed in order to deal with the derivatives of the auxiliary function $F$.}
\label{fig:eigen}
\end{figure}\\
The choice of $F$ immediately implies
\begin{align*}
\int   F|\nabla u|^2+ 4\tau \nabla u \cdot \nabla^2 \phi \nabla u dxdt \geq \frac{2}{5} \int |\nabla u|^2dxdt.
\end{align*}
As the weight function satisfies the pseudoconvexity condition, we also obtain
\begin{align*}
4 \tau^3 \int  \nabla \phi \cdot \nabla^2 \phi \nabla \phi u^2 + \lambda_{\min}(\nabla^2 \phi)|\nabla \phi|^2 u^2 dxdt\geq 0.
\end{align*}
As a consequence, it remains to prove the positivity of the following terms
\begin{align*}
& \int (-2\tau^2 \dt|\nabla \phi|^2 +\tau  \dt^2 \phi - \tau  \D^2\phi - \frac{2\tau^2}{5}|\nabla \phi|^2)u^2 dxdt\\
& - \int (\frac{1}{2} \D F  - \tau  \dt \phi F)u^2 dxdt - \frac{1}{4}\int F^2u^2dx dt .
\end{align*}

We begin with the terms of order $\tau^2$ and treat the terms involving $\phi_{1}$ and $\phi_{2}$ separately:
We start by estimating the $\phi_{1}$ contributions. Moreover, we note that by choosing $R\gg 1$ sufficiently large, the scaling of $\lambda_{\min}(\nabla^2 \phi)$ in the radial variable implies that the $\tau^2$ contribution coming from the $\frac{1}{4}\int F^2u^2dx dt $ integral can be considered small with respect to the other terms of $\tau^2$ scaling. Thus, it will be ignored in the sequel.
Due to the homogeneity of the remaining terms in the radial variable ($-2\tau^2 \dt|\nabla \phi|^2   - 4 \tau^2  \dt \phi_{1} \lambda_{\min}(\nabla^2 \phi)- \frac{2\tau^2}{5}|\nabla \phi|^2 \sim r^{2\alpha -2}$, $\alpha = 1.999999$) and the multiplicative temporal dependence of the weight, the lower bound
\begin{equation}
\begin{split}
\label{eq:sec}
 \int (-2\tau^2 \dt|\nabla \phi|^2  -  4\tau^2  \dt \phi_{1} \lambda_{\min}(\nabla^2 \phi)- \frac{2\tau^2}{5}|\nabla \phi|^2)u^2 dxdt\\ \geq c\tau^2 \int \frac{(1-t)}{t^2}u^2dxdt
\end{split} 
\end{equation}
follows, once it is established in an (angular) cross-section of the domain. 
In order to deduce this estimate, we observe
\begin{align*}
\lambda_{\min}(\nabla^2 \phi) \dt \phi_{1} & = - \frac{1-t}{t^2} \phi_{\alpha, \beta}(x) \lambda_{\min}(\nabla^2 \phi_{\alpha,\beta}) - \frac{(1-t)^2}{t^3} \phi_{\alpha,\beta}\lambda_{\min}(\nabla^2 \phi_{\alpha, \beta}),\\
\dt |\nabla \phi|^2 &= -2 \frac{1-t}{t^2}|\nabla \phi_{\alpha, \beta}|^2 - 2 \frac{(1-t)^2}{t^3}|\nabla \phi_{\alpha,\beta}|^2,
\end{align*}
where $\alpha = 1.9999999$, $\beta= 2.474917$.
Thus, it suffices to prove the positivity of 
\begin{align*}
3.6|\nabla \phi_{\alpha, \beta}|^2 + 4\lambda_{\min}(\nabla^2 \phi_{\alpha, \beta})\phi_{\alpha, \beta}.
\end{align*}
As this expression attains a local minimum at $\va=0$ and is non-negative at that point, the desired positivity follows, c.f. Fig. \ref{fig:sec}.
\begin{figure}[h]
\centering
\includegraphics[scale=0.6]{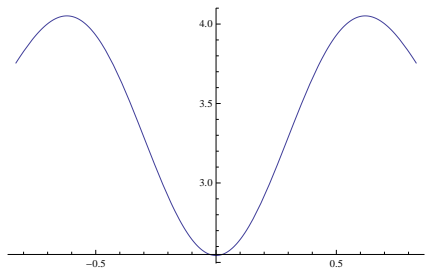} 
\caption{The figure depicts the term $3.6|\nabla \phi_{\alpha, \beta}|^2   + 4  \lambda_{\min}(\nabla^2 \phi_{\alpha,\beta})\phi_{\alpha, \beta}$ with $\alpha = 1.9999999$, $\beta= 2.474917$  in an angular cross-section of the domain. The numerical evaluation shows that this expression is positive.}
\label{fig:sec}
\end{figure}
In order to estimate the full contribution in $\tau^2$, it remains to bound
\begin{align*}
- 4\tau^2 \int \dt \phi_{2}\lambda_{\min}(\nabla^2 \phi) u dxdt = -8\tau^2 \epsilon \int (1-t) \lambda_{\min}(\nabla^2 \phi)u^2 dxdt. 
\end{align*}
For sufficiently small $\epsilon$ this can be absorbed into the right hand side of (\ref{eq:sec}).\\
We proceed with the terms of $\tau$-scaling. 
Due to the positivity of $\dt^2\phi_{i}$, the estimate $|\dt^2 \phi_{i} | \geq |\dt \phi_{i}|$, $i\in \{1,2\}$, and the scaling in the radial direction, the last term,
\begin{align*}
\tau \int (\dt^2 \phi - \D^2\phi - \frac{\tau^{-1}}{2}\D F+ \frac{2}{5}\dt \phi + \frac{2}{5}\lambda_{\min}(\nabla \phi_{1}))u^2 dxdt, 
\end{align*}
is positive in the spatial interior of the domain (which in particular includes the time slice $t=1$) if a sufficiently large ball around the origin is excluded. Close to the spatial boundary the scaling of the involved terms allows to absorb the (potentially) negative parts, i.e.
\begin{align*}
\tau \int (- \D^2\phi - \frac{\tau^{-1}}{2}\D F + \frac{2}{5}\lambda_{\min}(\nabla \phi_{1}) )u^2 dxdt, 
\end{align*}
into (\ref{eq:sec}) for sufficiently large $\tau \geq \tau_{0}$ (here it is possible to ignore the also potentially negative contribution $\frac{2}{5}\dt \phi$ as it can always be absorbed into the larger positive term $\dt^2 \phi$).\\
Furthermore, a small amount of the $\dt^2 \phi_{1}$ contribution suffices to control the negative $\dt \phi_{1}$ derivative. Hence, we obtain a further positive contribution of the form
\begin{align*}
\tau \int (\dt^2 \phi + \frac{2}{5}\dt \phi)u^2 dxdt \gtrsim \tau \int u^2 dxdt.
\end{align*}
For sufficiently large $\tau$, this contribution can then be used to absorb the last negative term: $-\frac{1}{25}\int u^2 dxdt $.
\end{proof}

\section{Proof of the Backward Uniqueness Result}
\label{sec:bur}
Due to Lemma \ref{lem:Sverak} null-controllable solutions of the heat equation have exponential decay: 
For solutions of (\ref{eq:heat}) the estimate of \v{S}ver\'ak, Seregin and Escauriaza yields $|u|\leq Ce^{-c \frac{ dist(x,\partial \Omega)^{2} }{t}}$ -- which is, at first sight, only a \textit{non-uniform} decay estimate, deteriorating close to the boundary of the domain. Considering angles strictly larger than $90^{\circ}$, it is possible to reduce the angle slightly while still remaining arbitrarily close to the original angle. In this case \v{S}ver\'ak's inequality implies a \textit{uniform} decay estimate:
\begin{align*}
|u|\leq C e^{-c |x|^{2}}.
\end{align*}

In order to deduce the backward uniqueness property, we use the following strategy of proof:
\begin{itemize}
\item In the first step the angle is reduced slightly, so as to obtain the Gaussian decay estimate globally.
\item Secondly, with the aid of a cut-off function, which is active at the boundary of the domain as well as close to the (spatial) origin, the Carleman estimate can be applied. 
\item Finally, carrying out the limit $\tau \rightarrow \infty$ provides the desired conclusion.
\end{itemize}

\begin{proof}[Proof of Proposition \ref{prop:backwarduniqueness}]
Step 1: \textit{ Decay estimate, rescaling and choice of the test functions.}\\

We choose $\epsilon>0$ such that $\delta:=\theta-\epsilon \geq \theta_{0}$ where $\theta_0$ is the angle down to which our Carleman inequalities hold (e.g. $\theta_0= 95.4^{\circ}$). Lemma \ref{lem:Sverak} implies Gaussian decay:
\begin{align*}
|u|\leq Ce^{-c \frac{ dist(x,\partial \Omega_{\theta})^{2} }{t}}  \leq Ce^{-c \frac{ sin^2(\epsilon/2)|x|^{2} }{t}}.
\end{align*}

 Due to the assumptions, we have to deal with the differential inequality
\begin{align*}
|(\dt + \D)u| \leq C(|u|+|\nabla u|).
\end{align*}
Rescaling $ u$ parabolically and translating, i.e. 
\begin{align*}
u_{\lambda}(x,t):= u\left(\lambda^{2}\left(t- \frac{1}{2}\right),\lambda x \right),
\end{align*}
we obtain a ``small" right hand side:
\begin{equation}
\begin{split}
\label{eq:lambda}
|(\dt + \D)u_{\lambda}| &\leq C\lambda (|u_{\lambda}|+|\nabla u_{\lambda}|) \mbox{ in } [0,1]\times \Omega_{\delta},\\
u_{\lambda}(t,x) &= 0 \mbox{ in } \left(0,\frac{1}{2}\right)\times \Omega_{\delta}.
\end{split}
\end{equation}
With slight abuse of notation, we will work with $u$ satisfying (\ref{eq:lambda}) in the sequel without changing notation.\\

In order to apply the Carleman estimate, it is necessary to cut off the solution of the heat equation. Due to this, we introduce the cut-off functions 
\begin{align*}
w_{1,R}(x_{1}):=  \left\{ \begin{array}{ll}
				0, & x_{1}\leq R,\\
				1, & x_{1}\geq 2R,
			\end{array} \right.   \ \
w_{2}(s) :=  \left\{ \begin{array}{ll}
				0, & s \leq  - \frac{4}{3} ,\\
				1, & s \geq - \frac{1}{2},
			\end{array} \right.
\end{align*}
which are chosen to be smooth interpolations in the intermediate regime.
Furthermore, we define
\begin{align*}
&w(x,t):= w_{1,R}(x_{1})w_{2}(\phi(t,x)-C),\\
&v(x,t) := w(x,t)u(x,t).
\end{align*}
Although $v$ does not have compact support, an additional limiting argument combined with the Gaussian decay rate of this function, implies its admissibility in the Carleman estimate.\\

Step 2: \textit{Application of the parabolic Carleman inequality and limit} $\tau \rightarrow \infty$.
An application of the Carleman inequality (\ref{eq:Carlm}) leads to
\begin{align}
\label{eq:Carl}
\tau^{\frac{1}{2}} \left\| e^{\tau (\phi-C)}v\right\|_{L^{2}} + \left\| e^{\tau( \phi-C)}\nabla v\right\|_{L^{2}} \leq \left\| e^{\tau (\phi-C)}(\dt+ \D)v\right\|_{L^{2}}.
\end{align}
Estimating the right hand side results in
\begin{align*}
|(\dt+\D)v| \leq C\lambda(|v|+ |\nabla v|) + C(|u|+|\nabla u|)(|\dt w| + |\nabla w| + |\D w| );
\end{align*}
Due to the smallness of $\lambda$, the first part of the expression can be absorbed in the left hand side of (\ref{eq:Carl}). For the remaining part, i.e. $ C(|u|+|\nabla u|)(|\dt w| + |\nabla w| + |\D w|) $, we use the definition of $w$.  Indeed, in the set on which $ C(|\dt w| + |\nabla w| + |\D w|) \neq 0$, we have $\phi-C\leq -\frac{1}{2}$. Consequently, 
\begin{align*}
&C \left\| e^{\tau (\phi-C)}(|u|+|\nabla u|)(|\dt w| + |\nabla w| + |\D w|)|\right\|_{L^{2}}\\
 & \leq C \left\| e^{-\frac{\tau}{2}}(|u|+|\nabla u|)(|\dt w| + |\nabla w| + |\D w|)|\right\|_{L^{2}}\\
& \leq  C \left\| e^{-\frac{\tau}{2}-\beta |x|^{2}}P_{\phi}(x)\right\|_{L^{2}},
\end{align*}
where $P_{\phi}(x)$ has at most polynomial growth. In the limit $\tau \rightarrow \infty$ the right hand side of the inequality vanishes. Hence,
\begin{align*}
\tau^{\frac{1}{2}} \left\| e^{\tau (\phi-C)}v\right\|_{L^{2}} + \left\| e^{\tau( \phi-C)}\nabla v\right\|_{L^{2}}& \leq C \left\| e^{-\frac{\tau}{2}}e^{-\beta |x|^{2}} P_{\phi}(x) \right\|_{L^{2}} \rightarrow 0 \mbox{ as } \tau \rightarrow \infty.
\end{align*}
As a result,
\begin{align*}
u = 0 \mbox{ in } \Omega_{\delta}\cap\{\phi-C \geq 0\}.
\end{align*}
Now unique continuation across spatial boundaries implies the desired result.
\end{proof}

\section[Reduction to the Elliptic Problem]{Backward Uniqueness by Reduction to the Elliptic Problem}
\label{sec:proI}

In this section we illustrate how the uniqueness of parabolic controls can be reduced to an elliptic problem. For this elliptic problem we prove a Carleman estimate.

\subsection{The Elliptic Pseudoconvexity Analysis}
As before, a key step in proving the desired Carleman estimates consists of the analysis of the conjugated operator. In choosing our weight, we consider a general weight of the form $\tau \phi$. Calculating the conjugated operator yields
\begin{align*}
L_{\phi} = e^{\tau \phi}\D e^{-\tau \phi} = \D - 2\tau \nabla \phi \cdot \nabla + \tau^2 |\nabla \phi|^2 - \tau \D \phi.
\end{align*}
The symmetric and antisymmetric parts of this operator are given by
\begin{align*}
S_{\phi} &= \D + \tau^2 |\nabla \phi|^2,\\ 
A_{\phi} &= -2\tau \nabla \phi \cdot \nabla - \tau \D \phi.
\end{align*}
We remark that although the original operator was elliptic, the resulting symmetric and antisymmetric parts of the conjugated operator are not elliptic anymore. Expanding the $L^{2}$ norm of $L_{\phi}$, we thus infer
\begin{align*}
\left\| L_{\phi} w \right\|_{L^{2}(\Omega)}^2 = \left\| S_{\phi} w\right\|_{L^{2}(\Omega)}^2 + \left\| A_{\phi} w\right\|_{L^{2}(\Om)}^2 + \int\limits_{\Om}([S_{\phi}, A_{\phi}]w,w)dx,
\end{align*}
for $w\in C_{0}^{\infty}(\Omega)$. Hence, on the intersections of the characteristic sets of the symmetric and antisymmetric parts, the necessary amount of positivity has to originate from the commutator:
\begin{align*}
 \int\limits_{\Om}([S_{\phi}, A_{\phi}]w,w)dx = \int\limits_{\Om}4\tau^3 \nabla \phi\cdot \nabla^2 \phi \nabla \phi w^2 + 4\tau \nabla w\cdot \nabla^2 \phi \nabla w - \tau \D^2 \phi w^2 dx.
\end{align*} 
In order to understand the behaviour of this expression, it is helpful to switch to a microlocal point of view. The principal symbols of the symmetric, antisymmetric and commutator part turn into
\begin{align*}
p^{r} &= -|\xi|^2 + \tau^2|\nabla \phi|^2,\\
p^{i} &= -2\tau \nabla \phi \cdot \xi,\\
\{p^r,p^i\} &= 4(\tau^3 \nabla \phi \cdot \nabla^2 \phi \nabla \phi + \tau \xi\cdot \nabla^2 \phi \xi).
\end{align*}
Therefore, the intersection of the characteristic set of the symmetric and antisymmetric parts of the operator is given by
\begin{align*}
\{ |\xi|^2 = \tau^2 |\nabla \phi|^2 \} \cap \{\nabla \phi \cdot  \xi =0\}.
\end{align*}
In the two-dimensional setting this leads to simplifications in the Poisson bracket:
\begin{align*}
\{p^r,p^i\} = 4\tau^3 \D \phi|\nabla \phi|^2 \mbox{ in } \{ |\xi|^2 = \tau^2 |\nabla \phi|^2 \} \cap \{\nabla \phi \cdot  \xi =0\}.
\end{align*}
Hence, the corresponding pseudoconvexity condition for the weight turns into subharmonicity:
\begin{align*}
\D \phi \geq 0  \mbox{ in } \{ |\xi|^2 = \tau^2 |\nabla \phi|^2 \} \cap \{\nabla \phi \cdot  \xi =0\}.
\end{align*}
In the sequel we will construct weights satisfying this property with sufficient decay in infinity. These are exactly the ``limiting Carleman weights" of Kenig et al. \cite{DKSU}, \cite{KSU}.

\subsection{Carleman Inequalities for the Laplacian in Conical Domains}

Before turning to the proof of Proposition \ref{prop:backwaruniquness-}, we first focus on Carleman inequalities for the Laplacian on conical domains. For this purpose, we use weights which are concentrated in the interior of the domain and vanish on the boundary --  the necessity of this stems form the lack of control of the boundary and initial data. The explicit choice of the weight is motivated by the requirement of satisfying the elliptic pseudoconvexity condition -- which amounts to a considerably easier condition than the corresponding parabolic analogue.\\

We prove the Carleman estimate by rescaling a local estimate. As the weight which we use satisfies a strict pseudoconvexity condition on $\Omega_{\theta}\setminus B_{1}$, the symbol calculus directly implies the estimate

\begin{prop}
\label{prop:Carlloc}
Let $\Omega_{\theta} \subset \R^{2}$ be the conical domain defined above with $\frac{\pi}{2}<\theta < \pi $. Let $\phi(x,y)= \Re((x+iy)^{\alpha}) + \epsilon x^{\alpha}$, with $\alpha = \frac{\pi}{\theta}$, $\epsilon>0$ arbitrary. Then  for $\tau\geq \tau_0 >0$ it holds
\begin{align*}
\tau^{3}\left\| e^{\tau\phi}|x|^{\frac{3 \alpha -4}{2}} u \right\|_{L^{2}(\Omega_{\theta}\cap (B_{2}\setminus B_{1}))}^2 + \tau \left\|  e^{\tau \phi}|x|^{\frac{\alpha -2}{2}}\nabla u\right\|_{L^{2}(\Omega_{\theta}\cap (B_{2}\setminus B_{1}))}^2\\
\lesssim \left\| e^{\tau \phi} \Delta u \right\|_{L^{2}(\Omega_{\theta}\cap (B_{2}\setminus B_{1}))}^{2}
\end{align*}
for all $u \in C_{0}^{\infty}(\Omega_{\theta}\cap (B_{2}\setminus B_{1}))$.
\end{prop}

\begin{proof}[Proof of Proposition \ref{prop:Carlloc}] This follows immediately from a pseudoconvexity analysis, see for example \cite{Tu1}, \cite{Tat}. \end{proof}

With this and the scaling properties of the weight, the global estimate can be obtained via a decomposition and rescaling procedure.

\begin{prop}
\label{prop:Carl1}
Let $\Omega_{\theta} \subset \R^{2}$ be the conical domain defined above with $\frac{\pi}{2}<\theta < \pi $. Set $\phi(x,y)= \Re((x+iy)^{\alpha}) + \epsilon x^{\alpha}$, with $\alpha = \frac{\pi}{\theta}$, $\epsilon>0$ arbitrary. Then for $\tau \geq \tau_0 >0$ we have
\begin{align}
\label{eq:Carlel}
\tau^{3}\left\| e^{\tau\phi} |x|^\frac{3\alpha -4}{2} u \right\|_{L^{2}(\Omega_{\theta}\setminus B_{1})}^2 + \tau \left\| e^{\tau \phi}|x|^{\frac{\alpha-2}{2}}\nabla u\right\|_{L^{2}(\Omega_{\theta}\setminus B_{1})}^2
\lesssim \left\| e^{\tau \phi} \Delta u \right\|_{L^{2}(\Omega_{\theta}\setminus B_{1})}^{2}
\end{align}
for all $u \in C_{0}^{\infty}(\Omega_{\theta}\setminus B_{1}(0)).$
\end{prop}

\begin{proof}[Proof of Proposition \ref{prop:Carl1}]
Using a decomposition of space, a scaling argument yields the claim:
We decompose $u= \sum\limits_{i\in \N}u_{i}$, $u_{i}(x):= (u\eta_{i})(x):= u(x) \eta(\frac{|x|-2^{i}}{2^{i}})$, where 
$\supp(\eta)\subset (0.5,2.5)$, i.e. $\eta$  
is a cut-off function normalized so as to provide a partition of unity. Setting $v_{i}(x):= u_{i}(2^{i}x)$, we obtain
\begin{align*}
\tau^{3}\left\| e^{\tau\phi} |x|^\frac{3\alpha -4}{2} u \right\|_{L^{2}(\Omega_{\theta}\setminus B_{1})}^2 
& \lesssim \tau^{3}\sum\limits_{i\in \N}\left\||x|^\frac{3\alpha -4}{2} e^{\tau\phi} u_{i} \right\|_{L^{2}((\Omega_{\theta}\cap B_{2^{i+1}})\setminus B_{2^{i}})}^2 \\
& = \tau^{3}\sum\limits_{i\in \N} \left\| |2^{i}x|^\frac{3\alpha -4}{2} e^{\tau\phi(2^{i}x)} (\eta u)(2^{i}x) \right\|_{L^{2}((\Omega_{\theta}\cap B_{2 })\setminus B_{1})}^2 2^{in}\\
& \lesssim  \sum\limits_{i \in \N} \tilde{\tau}^{3}2^{-4i}\left\| e^{\tilde{\tau}\phi(x)} v_{i}(x) \right\|_{L^{2}((\Omega_{\theta}\cap B_{2 })\setminus B_{1})}^2 2^{in}\\
& \stackrel{Prop. \ref{prop:Carlloc}}{\lesssim} \sum\limits_{i\in \N} 2^{-4i}\left\| e^{\tilde{\tau}\phi(x)}  \Delta v_{i}(x) \right\|_{L^{2}((\Omega_{\theta}\cap B_{2 })\setminus B_{1})}^2 2^{in}\\
& \lesssim \sum\limits_{i\in \N}  \left\| e^{\tau\phi(x)} \Delta u_{i}(x) \right\|_{L^{2}((\Omega_{\theta}\cap B_{2^{i+1} })\setminus B_{2^{i}})}^2 \\
& \lesssim \left\| e^{\tau \phi} \Delta u \right\|_{L^{2}(\Omega_{\theta}\setminus B_{1})}^{2} \\
& \;\;\;\;\;+  \left\| e^{\tau \phi}|x|^{-1} |\nabla u| \right\|_{L^{2}(\Omega_{\theta}\setminus B_{1})}^{2}
+ \left\| e^{\tau \phi}|x|^{-2} u \right\|_{L^{2}(\Omega_{\theta}\setminus B_{1})}^{2},
\end{align*}
where we used the notation $\tilde{\tau}= 2^{i\alpha}\tau$.
In this estimate the last terms are error terms originating from the partition of unity. These will be absorbed in the left hand side for sufficiently large $\tau$.

Analogously, the result for the gradient term can be derived:
\begin{align*}
\tau \left\| e^{\tau \phi}|x|^{\frac{\alpha-2}{2}}\nabla u\right\|_{L^{2}(\Omega_{\theta}\setminus B_{1})}^2
& \lesssim \tau\sum\limits_{i\in \N}\left\| e^{\tau\phi} |x|^{\frac{\alpha-2}{2}} |\nabla u_{i}| \right\|_{L^{2}(\Omega_{\theta}\cap B_{2^{i+1}}\setminus B_{2^{i}})}^2 \\
& = \tau \sum\limits_{i\in \N} \left\| e^{\tau\phi(2^{i}x)} |2^{i}x|^{\frac{\alpha-2}{2}} |\nabla (u \eta)(2^{i}x)| \right\|_{L^{2}((\Omega_{\theta}\cap B_{2 })\setminus B_{1})}^2 2^{in}\\
& \lesssim \tilde{\tau} \sum\limits_{i \in \N} 2^{-4i} \left\| e^{\tilde{\tau}\phi(x)} |x|^{\frac{\alpha-2}{2}}|\nabla v_{i}(x)| \right\|_{L^{2}((\Omega_{\theta}\cap B_{2 })\setminus B_{1})}^2 2^{in}\\
& \stackrel{Prop. \ref{prop:Carlloc}}{\lesssim} \sum\limits_{i \in \N} 2^{ -4i}\left\| e^{\tilde{\tau}\phi(x)} \Delta v_{i}(x) \right\|_{L^{2}((\Omega_{\theta}\cap B_{2 })\setminus B_{1})}^2 2^{in}\\
& \lesssim \sum\limits_{i\in \N}  \left\| e^{\tau\phi(x)} \Delta u_{i}(x) \right\|_{L^{2}((\Omega_{\theta}\cap B_{2^{i+1} })\setminus B_{2^{i}})}^2 \\
& \lesssim \left\| e^{\tau \phi} \Delta u \right\|_{L^{2}(\Omega_{\theta}\setminus B_{1})}^{2}
\\
& \;\;\;\;\; +  \left\| e^{\tau \phi}|x|^{-1} |\nabla u| \right\|_{L^{2}(\Omega_{\theta}\setminus B_{1})}^{2}
+ \left\| e^{\tau \phi}|x|^{-2} u \right\|_{L^{2}(\Omega_{\theta}\setminus B_{1})}^{2}.
\end{align*} 
Adding both inequalities and noting $\frac{3 \alpha -4}{2}\geq -2$, $\frac{\alpha -2}{2}\geq -1$ for $\alpha \geq 0$, the error terms can be absorbed. This yields the desired estimate.
\end{proof}

\subsection[Proof of Proposition \ref{prop:backwaruniquness-}]{Proof of Proposition \ref{prop:backwaruniquness-} and an Non-Existence Proof of Harmonic Functions with Gaussian Decay Rates in Cones with Opening Angles Larger than $\frac{\pi}{2}$}
\label{sec:non}

In the sequel we assume $\alpha > \frac{4}{3}$, which translates into a condition on the opening angle of the domain: $\theta < \frac{3 \pi}{4}$. The backward uniqueness result for conical domains with larger opening angles immediately follows from this by restriction.
It is deduced from the \textit{elliptic} Carleman estimates by an application of the Laplace or a one-sided Fourier transform. Indeed, the $t$-independence of the coefficients of equation (\ref{eq:heat-}) and \v{S}ver\'ak's decay result, Lemma \ref{lem:Sverak}, lead to 
\begin{align*}
s\La{u}(s,x) + \D\La{u}(s,x) & = c_{1}(x)\La{u}(s,x) + c_{2}(x) \cdot \nabla\La{u}(s,x) \mbox{ in } \R\times \Omega_{\theta},\\
|\La{u}| & \leq Ce^{-\beta |x|^2} \mbox{ in } \R \times \Omega_{\theta}.
\end{align*} 
The backward uniqueness result is derived as a consequence of a Carleman estimate -- more precisely, of the elliptic estimate (\ref{eq:Carlel}). Keeping $s$ fixed and rescaling in $x$, it is possible to assume the ``smallness" condition:
\begin{equation}
\label{eq:uniq}
\begin{split}
| \D\La{u}(s,x)| & \leq \lambda(|\tilde{c}_{1}(x)||\La{u}(s,x)| + |c_{2}(x)||\nabla \La{u}(s,x)|) \mbox{ in } \R\times \Omega_{\theta},\\
|u| & \leq Ce^{-\beta \lambda^2 |x|^2} \mbox{ in } \R\times \Omega_{\theta},
\end{split}
\end{equation} 
with $\lambda \leq 1$ and $\tilde{c}_{1}=c_{1}+|s|$.
Using (smooth) cut-off functions which satisfy the following limiting behaviour
\begin{align*}
w_{1,R}(x_{1})&:=  \left\{ \begin{array}{ll}
				0, & x_{1}\leq R,\\
				1, & x_{1}\geq 2R,
			\end{array} \right.
w_{2}(s) :=  \left\{ \begin{array}{ll}
				0, & s \leq  - \frac{4}{3} ,\\
				1, & s \geq - \frac{1}{2},
			\end{array} \right.
\eta_{L}(r):=  \left\{ \begin{array}{ll}
				1, & r\leq L,\\
				0, & r\geq 2L,
			\end{array} \right. 
\end{align*}
we insert $v_{R,L}(x):=\La{u}(x)w_{1,R}(x_{1})w_{2}(\phi)\eta_{L}(|x|)$ into the Carleman inequality (\ref{eq:Carlel}). Recalling the decay condition on the function $\La{u}$ and invoking the dominated convergence theorem, we can pass to the limit $L\rightarrow \infty$. For $v_{R}:= \La{u}w_{1,R}w_{2}(\phi)$ we then obtain
\begin{align*}
\tau^{3}\left\| e^{\tau(\phi-C)} |x|^\frac{3\alpha -4}{2} v_{R} \right\|_{L^{2}(\Omega_{\theta}\setminus B_{1})}^2 + \tau \left\| e^{\tau (\phi-C)}|x|^{\frac{\alpha-2}{2}}\nabla v_{R} \right\|_{L^{2}(\Omega_{\theta}\setminus B_{1})}^2 \\
\lesssim \left\| e^{\tau (\phi-C)} \Delta v_{R} \right\|_{L^{2}(\Omega_{\theta}\setminus B_{1})}^{2}.
\end{align*}
Defining $\tilde{w}:= w_{1,R}w_{2}$, we inspect the right hand side of the inequality:
\begin{align*}
\D v_{R} = \tilde{w} \D \La{u} + 2 \nabla \tilde{w}\cdot \nabla \La{u} + \La{u} \D \tilde{w}.
\end{align*}
Combining this with inequality (\ref{eq:uniq}) and choosing $R\geq 1$ sufficiently large, yields
\begin{align}
\label{eq:error}
|\D v_{R}| \leq C\lambda (|x|^{\frac{3\alpha-4}{2}}|v_{R}|+  |x|^{\frac{\alpha-2}{2}}|\nabla v_{R}|) + 2 |\nabla \tilde{w}|| \nabla \La{u}| +  |\La{u}| |\D \tilde{w}|.
\end{align}
Thus, the first term on the right hand side can be absorbed into the left hand side of the Carleman inequality. The remaining right hand side terms in (\ref{eq:error}) are only active close to the boundary as well as at a spatial scale $\sim R$. With $R\sim 1$ and choosing $C>0$ in dependence of $R$, this leads to a right hand side term of the form
\begin{align*}
\left\| e^{\tau (\phi-C)} ( |\nabla \tilde{w}|| \nabla \La{u}| +  |\La{u}| |\D \tilde{w}| ) \right\|_{L^{2}(\Omega_{\theta}\setminus B_{1})}
\lesssim \left\| e^{-\tau \frac{C}{2}}e^{- \beta |x|^2} P_{\phi}(x) \right\|_{L^{2}(\Omega_{\theta}\setminus B_{1})},
\end{align*}
where $P_{\phi}$ denotes a function with at most polynomial growth. As a consequence, the right hand side term vanishes in the limit $\tau \rightarrow \infty$. Thus, the function $\La{u}$ must vanish on some open domain. By unique continuation this therefore implies that $\La{u}\equiv 0$ in the whole domain.
As this holds for all Laplace modes $s$, we obtain the desired result $\La{u} \equiv 0$, hence $u \equiv 0$.

\begin{rmk}
The Carleman estimate (\ref{eq:Carlel}) dictates the decay assumption on the potential $c_2$. Comparing exponents, we obtain
\begin{align*}
b(\theta) = \frac{\alpha-2}{2}
\end{align*}
for the exponent in Proposition \ref{prop:backwaruniquness-}.
\end{rmk}

We point out that the non-existence of harmonic functions with Gaussian decay rates in cones with opening angles greater than $\frac{\pi}{2}$ follows via our Carleman inequality (\ref{eq:Carlel}):

\begin{prop}
\label{prop:non-existence}
Let $\theta > \frac{\pi}{2}$.
Let $u:\Omega_{\theta}\rightarrow \R$ be a solution of
\begin{align*}
\D u &= c_{1}(x) u + c_{2}(x) \cdot \nabla u \mbox{ in } \Omega_{\theta},\\
|c_{2}(x)| & \leq \frac{1}{|x|^{b(\theta)}} \mbox{ in } \Omega_{\theta}, \ c_{1}\in L^{\infty},\\
|u| & \leq e^{-\beta|x|^2} \mbox{ in } \Omega_{\theta}.
\end{align*}
Then $u\equiv 0$.
\end{prop}

\begin{proof}
This follows along the lines of the proof of Proposition \ref{prop:backwaruniquness-} (after having carried out the Laplace transform). 
\end{proof}

\bibliography{citations1,citations}
\bibliographystyle{plain}

\end{document}